\documentclass[12pt]{article}
\usepackage{amsmath, amssymb, amsthm, mathabx}
\usepackage[colorlinks=true]{hyperref}

\theoremstyle{plain}
\newtheorem{theorem}{Theorem}
\newtheorem{corollary}[theorem]{Corollary}
\newtheorem{lemma}[theorem]{Lemma}

\theoremstyle{definition}
\newtheorem{definition}[theorem]{Definition}

\theoremstyle{remark}

\newcommand{\N}{\mathbb{N}}
\newcommand{\Z}{\mathbb{Z}}
\newcommand{\ds}{\displaystyle}

\newcommand{\Lfloor}{\left\lfloor}
\newcommand{\Rfloor}{\right\rfloor}

\allowdisplaybreaks[1]

\date{}

\begin{document}

\title{Congruences of Power Sums}
\markright{Symmetry of the Power Sum Polynomials}

\author{Nicholas J. Newsome, Maria S. Nogin,\\ and Adnan H. Sabuwala\\Department of Mathematics\\California State University, Fresno\\Fresno, CA 93740 \\ USA\\\href{mailto:mrgoof@mail.fresnostate.edu}{mrgoof@mail.fresnostate.edu}\\\href{mailto:mnogin@csufresno.edu}{mnogin@csufresno.edu}\\\href{mailto:asabuwala@csufresno.edu}{asabuwala@csufresno.edu}}

\maketitle

\begin{abstract}
The following congruence for power sums, $S_n(p)$, is well known and has many applications: 
$$1^n+2^n +\dots +p^n \equiv\begin{cases} 
-1 \pmod{p}, & \text{ if } \ p-1 \divides n; \\ 
0 \pmod{p}, & \text{ if } \ p-1 \notdivides n, 
\end{cases}$$
where $n\in\N$ and $p$ is prime. We extend this congruence, in particular, to the case when $p$ is any power of a prime. We also show that the sequence $(S_n(m) \bmod{k})_{m \geq 1}$ is periodic and determine its period. 
\end{abstract}

\section{Introduction}

Sums of powers of integers defined below have captivated mathematicians for many centuries \cite{Beery}.

\begin{definition}
For $n,m\in\N$, let 
$$S_n(m) = \sum_{i=1}^m i^n.$$
\end{definition}

With their pebble experiments, the Pythagoreans were the first to discover a formula for the sum of the first powers. Formulas for the sums of second and third powers were proved geometrically by Aryabhatta and Archimedes, and Harriot later provided a generalizable form for these formulas. Faulhaber gave formulas for power sums up to the $17^{\text{th}}$ power, and Fermat, Pascal, and Bernoulli provided succinct formulas for them. Since then, different representations and number-theoretic properties of these power sums have been an object of study \cite{Knuth, MS}. Bernoulli numbers have been used to represent the coefficients of polynomial formulas for these power sums such as Faulhaber's formula \cite[p.\ 107]{Conway}. In a recent paper, Newsome et al.\ \cite{NMA} have demonstrated symmetry properties of the power sum polynomials and their roots via a novel Bernoulli number identity.

One of the most well-known results concerning the number-theoretic properties of power sums is the following congruence relation:

\begin{theorem}
\label{known-congruency-theorem}
If $n\in\N$ and $p$ is prime, then 
$$S_n(p)\equiv\begin{cases} 
-1 \pmod{p}, & \text{ if } \ p-1 \divides n; \\ 
0 \pmod{p}, & \text{ if } \ p-1 \notdivides n.
\end{cases}$$
\end{theorem}

The case $p-1 \divides n$ is an easy consequence of Fermat's little theorem. There are several different proofs for the case $p-1 \notdivides n$ in the literature. Some of the notable ones are by Rado \cite{HW, M2, R} using the theory of primitive roots, Zagier \cite{M} using Lagrange's theorem, and MacMillan and Sondow \cite{MS} using Pascal's identity. Also, a proof of both cases by Carlitz \cite{C} uses Bernoulli numbers.

This congruence is used to prove the von Staudt-Clausen theorem \cite{HW, R} and its generalization \cite{C}, prove the Carlitz-von Staudt theorem \cite {M}, and study the Erd\H{o}s-Moser equation $S_n(m-1)=m^n$ \cite{M, M2, Mo}. 

Our main goal in this paper is to generalize the well-known congruence result above and to present periodicity properties of the sequence $(S_n(m)\bmod{k})_{m\geq 1}$. In Section \ref{generalization-section}, we extend Theorem \ref{known-congruency-theorem} to the case when $p$ is a power of a prime. In Section \ref{periodicity-section}, we prove that $(S_n(m) \bmod{k})_{m\geq 1}$ is periodic and determine its period for different values of $k$ and $n$. 

\section{Generalization of Theorem \ref{known-congruency-theorem}}
\label{generalization-section}

\begin{theorem}
\label{generalized-congruency-theorem}
\begin{enumerate}
\item[(1)] For $n\in\N$ and $p=2^a$ with $a\ge 2$, 
$$S_n(p)\equiv\begin{cases} 
\varphi(p) \pmod{p}, & \text{ if } \ n=1 \text{ or } 2 \divides n; \\ 
0 \pmod{p}, & \text{ if } \ n>1 \text{ and } 2 \notdivides n,
\end{cases}$$
where $\varphi$ is Euler's totient function.

\item[(2)] If $n\in\N$ and $p=q^a$ where $q$ is an odd prime and $a\ge 1$, then 
$$S_n(p)\equiv\begin{cases} 
\varphi(p) \pmod{p}, & \text{ if } \ q-1 \divides n; \\ 
0 \pmod{p}, & \text{ if } \ q-1 \notdivides n.
\end{cases}$$
\end{enumerate}
\end{theorem}

\begin{proof} 
\begin{enumerate}
\item[(1)]
The proof is by induction on $a$. For $a=2$, 
\begin{align*}
S_n(4) & \equiv 1^n + 2^n + 3^n + 4^n \equiv 1^n + 2^n + (-1)^n \\
& \equiv \begin{cases}
2 \pmod{4}, & \text{ if $n=1$ or $2\divides n$}; \\
0 \pmod{4}, & \text{ if $n>1$ and $2\notdivides n$}.
\end{cases}
\end{align*}
Suppose the statement holds for some $a\ge 2$. Then, for $n=1$ we have 
$$S_1(2^{a+1}) \equiv \frac{2^{a+1}(2^{a+1}+1)}{2} \equiv 2^a \equiv \varphi(2^{a+1}) \pmod{2^{a+1}},$$ 
and for $n\ge 2$, 
\begin{align*}
S_n(2^{a+1}) & \equiv 1^n+\dots +(2^a)^n + (2^a+1)^n +\dots +(2^{a+1})^n \\
& \equiv S_n(2^a) + \sum_{t=1}^{2^a} (2^a+t)^n \\ 
& \equiv S_n(2^a) + \sum_{t=1}^{2^a} \left( t^n + n2^at^{n-1} \right)\\ 
& \text{(all other terms are divisible by $(2^{a})^2$, thus divisible by $2^{a+1}$)} \\ 
& \equiv 2S_n(2^a) + n2^a \underbrace{S_{n-1}(2^a)}_{\text{even}} \\
& (\text{since $a \ge 2$, $S_{n-1}(2^a)$ has an even number of odd terms})\\ 
& \equiv 2S_n(2^a) \pmod{2^{a+1}}. 
\end{align*}
If $2 \divides n$, then 
$$S_n(2^a)\equiv \varphi(2^a) \pmod{2^a},$$ 
so 
$$S_n(2^{a+1})\equiv2S_n(2^a)\equiv 2\varphi(2^{a}) \equiv \varphi(2^{a+1}) \pmod{2^{a+1}}.$$
If $2 \notdivides n$, then 
$$S_n(2^a)\equiv 0 \pmod{2^a},$$ 
so 
$$S_n(2^{a+1})\equiv2S_n(2^a)\equiv 0 \pmod{2^{a+1}}.$$ 

\item[(2)]
The proof is by induction on $a$. The case $a=1$ is Theorem \ref{known-congruency-theorem}. 

Suppose the statement holds for some $a\ge 1$. Then, 
for $n=1$ we have 
$$S_1(q^{a+1}) = \frac{q^{a+1}(q^{a+1}+1)}{2} \equiv 0 \pmod{q^{a+1}},$$ 
and for $n\ge 2$, 
\begin{align*}
S_n(q^{a+1}) & \equiv \left(1^n+\dots +(q^a)^n\right) + \dots + \left(((q-1)q^a+1)^n +\dots +(q^{a+1})^n\right) \\
& \equiv \sum_{i=0}^{q-1} \sum_{t=1}^{q^a} (iq^a+t)^n \\ 
& \equiv \sum_{i=0}^{q-1} \sum_{t=1}^{q^a} \left( t^n + niq^at^{n-1} \right) \\
& \text{(all other terms are divisible by $(q^{a})^2$, thus divisible by $q^{a+1}$)} \\ 
& \equiv \sum_{i=0}^{q-1} \left( S_n(q^a) + niq^aS_{n-1}(q^a) \right)\\ 
& \equiv qS_n(q^a) + n\frac{(q-1)q}{2}q^aS_{n-1}(q^a) \\ 
& \equiv qS_n(q^a) \pmod{q^{a+1}}. 
\end{align*}
If $q-1 \divides n$, then 
$$S_n(q^a)\equiv \varphi(q^a)  \pmod{q^a},$$ 
so 
$$S_n(q^{a+1})\equiv qS_n(q^a)\equiv q\varphi(q^a) \equiv \varphi(q^{a+1}) \pmod{q^{a+1}}.$$
If $q-1 \notdivides n$, then 
$$S_n(q^a)\equiv 0 \pmod{q^a},$$ 
so 
$$S_n(q^{a+1})\equiv qS_n(q^a)\equiv 0 \pmod{q^{a+1}}.$$ 
\end{enumerate}
\end{proof}

\begin{corollary}
\label{generalized-congruency-corollary}
For any $a,n\in\N$ and prime $q$,  
$$S_n(q^a)\equiv 0 \pmod{q^{a-1}}.$$
\end{corollary}

The next few results will be used to extend Theorem \ref{generalized-congruency-theorem} (2).

\begin{lemma}
\label{div-by-qpowers-lemma}
If $n,i,j,k\in\N$ and $q$ is an odd prime such that $q^i \divides n$, then $q^{i+j} \divides \ds\binom{n}{k}(q^j)^k$. Moreover, for $k\ge 2$, 
$q^{i+j+1} \divides \ds\binom{n}{k}(q^j)^k$.
\end{lemma}

\begin{proof}
If $k=1$, then $\ds \binom{n}{k}(q^j)^k = nq^j$ is divisible by $q^{i+j}$. 

If $k\ge 2$, note that since $\ds \binom{n}{k} = \frac{n(n-1)\cdots(n-k+1)}{k!}$ and the highest power of $q$ that divides $k!$ is $q^\alpha$ where $\ds \alpha=\Lfloor\frac{k}{q}\Rfloor+\Lfloor\frac{k}{q^2}\Rfloor+\cdots$, it is sufficient to show that 
$$i+j < i - \left( \Lfloor\frac{k}{q}\Rfloor+\Lfloor\frac{k}{q^2}\Rfloor+\cdots\right) +jk.$$

Indeed, 
\begin{align*}
i+j & \le i+j+\frac{k}{2}-1 \\
& = i+j-\frac{k}{2}+k-1 \\
& \le i+j- \frac{k}{q-1}+j(k-1) \\
& = i - \sum_{t=1}^{\infty}\frac{k}{q^t} +jk \\
& < i - \left( \Lfloor\frac{k}{q}\Rfloor+\Lfloor\frac{k}{q^2}\Rfloor+\cdots\right) +jk.
\end{align*}
\end{proof}

\begin{corollary}
\label{power-congruency-corollary}
If $i,j,n,t\in \N$, $q$ is an odd prime, and $q^i \divides n$, then 
$$(t+q^j)^n \equiv t^n\pmod{q^{i+j}}.$$
\end{corollary}

\begin{proof}
We have
\begin{align*}
(t+q^j)^n &\equiv t^n+\sum_{k=1}^{n} \binom{n}{k} (q^j)^k t^{n-k} \\
&\equiv t^n\pmod{q^{i+j}}. \qquad\text{(by Lemma \ref{div-by-qpowers-lemma})}
\end{align*}
\end{proof}

If $q$ is an odd prime and $g$ is invertible modulo $q$, then multiplication by $g$ permutes elements of $\Z_q^*$, that is, 
\begin{equation}\label{mult-by-g}
\{g\cdot1\bmod{q},\dots, g(q-1)\bmod{q}\}=\{1,\dots,q-1\}
\end{equation} 
as sets. 

The following theorem extends Theorem \ref{generalized-congruency-theorem} (2).

\begin{theorem}
\label{generator-block-theorem}
If $i\in\Z$, $i\ge 0$, $j,n\in\N$, $q$ is an odd prime, $q-1\notdivides n$, and $q^i \divides n$, then $$S_n(q^j)\equiv0\pmod{q^{i+j}}.$$ 
\end{theorem}

\begin{proof}
The case $i=0$ is Theorem \ref{generalized-congruency-theorem} (2). 

For any fixed $i\ge 1$ we use induction on $j$. First consider $j=1$. Let $g$ be a generator of the multiplicative group $\Z_{q}^*$. 
Then 
\begin{align*}
g^nS_n(q) &\equiv g^n \sum_{k=1}^{q}k^n \\
&\equiv \sum_{k=1}^{q} (gk)^n \\
&\equiv \sum_{k=1}^{q} (gk\bmod{q})^n \qquad \text{(by Corollary \ref{power-congruency-corollary})}\\
&\equiv \sum_{k=1}^{q}k^n \qquad \text{(by \eqref{mult-by-g})}\\ 
&\equiv S_n(q) \pmod{q^{i+1}}.
\end{align*}
Thus $$(g^n-1)S_n(q)\equiv0\pmod{q^{i+1}}.$$
But $g^n\not\equiv 1$ (mod $q$) since $g$ is a generator of $\Z_{q}^*$ and $q-1\notdivides n$. Therefore $$S_n(q)\equiv 0\pmod{q^{i+1}}.$$
Now assume that $S_n(q^j)\equiv 0$ (mod $q^{i+j}$) for some $j\ge1$. Then 
\begin{align*}
S_n(q^{j+1}) &\equiv \sum_{t=0}^{q-1}\sum_{r=1}^{q^j}(tq^j+r)^n\\
&\equiv \sum_{t=0}^{q-1}\sum_{r=1}^{q^j} \left( r^p+ntq^jr^{n-1}+\sum_{k=2}^n \binom{n}{k}(tq^j)^k r^{n-k}\right) \\
&\equiv \sum_{t=0}^{q-1}\sum_{r=1}^{q^j} \left( r^n+ntq^jr^{n-1}+0\right) \qquad\text{(by Lemma \ref{div-by-qpowers-lemma})} \\ 
&\equiv \sum_{t=0}^{q-1} \left( S_n(q^j) + ntq^j S_{n-1}(q^j)\right)\\
&\equiv qS_n(q^j) + n\frac{(q-1)q}{2}q^j S_{n-1}(q^j)\\
&\equiv 0\pmod{q^{i+j+1}}. \qquad (\text{since $q^i\divides n$})
\end{align*}
\end{proof}

\section{Periodicity}
\label{periodicity-section}

In this section, we first establish the periodicity of the sequence of sequences $((S_n(m)\bmod k)_{n\geq 1})_{m\geq 1}$ for any $k \in \N$. An immediate implication of this result is that the sequence $(S_n(m)\bmod k)_{m\geq 1}$ is periodic for all values of $k$ and $n$. We then provide formulas for the length of the period when $k$ is a power of a prime.

\begin{theorem}
\label{row-periodicity-theorem}
For each $k\in\N$, the sequence of sequences 
$$( (S_1(m) \bmod{k}, \ S_2(m) \bmod{k}, \ S_3(m) \bmod{k}, \ \dots ))_{m\geq 1}$$ 
is periodic. 
If $k=q_1^{a_1}q_2^{a_2}\cdots q_r^{a_r}$ where $q_i$'s are distinct primes, then the period is $q_1^{a_1+1}q_2^{a_2+1}\cdots q_r^{a_r+1}$.
\end{theorem}

\begin{proof} 
We will first prove that
$$S_n(m+q^{a+1}) \equiv S_n(m) \pmod{q^a}$$ 
for all prime $q$ and natural $n$, $m$, and $a$.  
We have 
\begin{align*}
S_n(m+q^{a+1}) & \equiv S_n(q^{a+1}) + (q^{a+1}+1)^n + \dots + (q^{a+1}+m)^n \\
& \equiv S_n(q^{a+1}) + S_n(m) \\ 
& \equiv S_n(m) \pmod{q^a}  
\end{align*}
since $S_n(q^{a+1})\equiv 0$ (mod $q^a$) by Corollary \ref{generalized-congruency-corollary}.\\
Thus, the sequence of sequences 
$$((S_1(m)\bmod{q^a}, \ S_2(m)\bmod{q^a}, \ S_3(m)\bmod{q^a}, \ \dots ))_{m\geq 1}$$ 
repeats every $q^{a+1}$ terms. Thus it is periodic with period being a factor of $q^{a+1}$. To show that the period is not less than $q^{a+1}$, it is sufficient to show that the sequence does not repeat every $q^a$ terms. More precisely, we will show that $S_n(q^a) \not\equiv S_n(q^{a+1})$ (mod $q^a$) for at least one value of $n$.

Consider $n=q-1$ (or, in fact, any $n$ divisible by $q-1$ if $q$ is odd). By Theorem \ref{known-congruency-theorem} in the case $a=1$, and by Theorem \ref{generalized-congruency-theorem} otherwise, and using Corollary \ref{generalized-congruency-corollary},  
$$S_n(q^a)\equiv \varphi(q^a) \not\equiv 0 \equiv S_n(q^{a+1})\pmod{q^a}.$$    
Thus the sequence does not repeat every $q^a$ terms, which implies the period is exactly $q^{a+1}$.\\
Next, if $k=q_1^{a_1}q_2^{a_2}\cdots q_r^{a_r}$ where $q_i$'s are distinct primes, then from the case proved above and the Chinese Remainder Theorem, it follows that the period of the sequence is $q_1^{a_1+1}q_2^{a_2+1}\cdots q_r^{a_r+1}$.
\end{proof}

It follows from Theorem \ref{row-periodicity-theorem} that given any values of $k$ and $n$, the sequence 
$$(S_n(m)\bmod{k})_{m\geq 1}$$ 
is periodic with period not exceeding the one given in Theorem \ref{row-periodicity-theorem}. 
However, for some values of $k$ and $n$ the period is smaller.
\begin{theorem}
\label{vertical-periodicity-theorem}
For $k,n\in \N$, let $\ell(k,n)$ denote the period of the sequence $(S_n(m)\bmod{k})_{m\geq 1}$. Then   
\begin{enumerate}
\item[(1)] $\ell(2,n)=4$ for all $n$.   

\item[(2)] for $a\ge 2$, 

$$\ell(2^a, n)=\begin{cases}
2^{a+1}, & \text{ if } n=1 \text{ or } 2 \divides n; \\
2^a, & \text{ otherwise}. 
\end{cases}$$

\item[(3)] for $q$ an odd prime and $a\ge 1$, 
$$\ell(q^a,n) = \begin{cases}
q^{a+1}, & \text{ if } q-1 \divides n; \\
q^{a-i}, & \text{ if } q-1\notdivides n, \ \nu_q(n) = i, \ 0 \le i \le a-2; \\
q, & \text{ if } q-1\notdivides n \text{ and } q^{a-1} \divides n, 
\end{cases}$$
where $\nu_q(n)$ is the exponent of the highest power of $q$ that divides $n$.
\end{enumerate}
\end{theorem}

\begin{proof}
\begin{enumerate}
\item[(1)] 
Theorem \ref{row-periodicity-theorem} implies that $\ell(2,n)$ is a factor of $4$. Since
$$\begin{array}{r@{\ }l}
1^n & \equiv 1 \pmod 2,\\
1^n+2^n & \equiv 1 \pmod 2,\\
1^n+2^n+3^n & \equiv 0 \pmod 2,\\
1^n+2^n+3^n+4^n & \equiv 0 \pmod 2,
\end{array}$$
$\ell(2,n)=4$.
\item[(2)]
Let $a \ge 2$.

If $n = 1$ or $2 \divides n$, by Theorem \ref{generalized-congruency-theorem} (1) we have $S_n(2^a) \equiv \varphi(2^a)$ (mod $2^a$). However, Theorem \ref{row-periodicity-theorem} implies that $\ell(2^a,n)$ is a factor of $2^{a+1}$, and hence must be $2^{a+1}$.

If $n > 1$ and $2 \notdivides n$, Theorem \ref{generalized-congruency-theorem} (1) implies that $S_n(2^a) \equiv 0$ (mod $2^a$).

We have 
\begin{align*}
S_n(m+2^{a}) & \equiv S_n(2^{a}) + (2^{a}+1)^n + \dots + (2^{a}+m)^n \\
& \equiv S_n(2^{a}) + S_n(m) \\ 
& \equiv S_n(m)  \pmod{2^a}.  
\end{align*}
Thus, $\ell(2^a,n)$ is a factor of $2^a$. We now show that $\ell(2^a,n)$ is not smaller than $2^a$. Assume to the contrary that $\ell(2^a,n)$ is a factor of $2^{a-1}$. Then $$S_n(2^{a-1})\equiv S_n(2^a)\equiv0\pmod{2^a},$$ but then
\begin{align*}
S_n(2^{a-1}+1) & \equiv S_n(2^{a-1}) + (2^{a-1}+1)^n \\
& \equiv 0 + 1^n + n2^{a-1} + \sum_{k=2}^n \binom{n}{k}(2^{a-1})^k 1^{n-k} \\
& \equiv 1^n + n2^{a-1}\\ 
& \not\equiv 1^n \qquad \text{(since $n$ is odd)} \\
& \equiv S_n(1) \pmod{2^a},
\end{align*}
which is a contradiction. 

\item[(3)] Let $q$ be an odd prime and $a\ge 1$. 

The case $q-1\divides n$ follows from the proof of Theorem \ref{row-periodicity-theorem}. 

If $q-1\notdivides n$ and $q^i\divides n$ for $0\le i\le a-1$, 
then by Theorem \ref{generator-block-theorem} 
$$S_n(q^{a-i})\equiv 0\pmod{q^a}.$$
Then 
\begin{align*}
S_n(m+q^{a-i}) & \equiv S_n(q^{a-i})+\sum_{r=1}^m(q^{a-i}+r)^n \\
& \equiv 0 + \sum_{r=1}^m \left(r^n +\sum_{k=1}^n \binom{n}{k} (q^{a-i})^k r^{n-k}\right)\\
& \equiv \sum_{r=1}^m r^n \qquad \text{(by Lemma \ref{div-by-qpowers-lemma})}\\
& \equiv S_n(m) \pmod{q^a}, 
\end{align*}
so $\ell(q^a,n)$ is a factor of $q^{a-i}$. 

We will show that if $q^{i+1}\notdivides n$ for $0\le i\le a-2$, then $\ell(q^a,n)$ is not smaller than $q^{a-i}$. Assume to the contrary that $\ell(q^a,n)$ is a factor of $q^{a-i-1}$. Then  
$$S_n(q^{a-i-1})\equiv S_n(q^{a-i}) \equiv 0 \pmod{q^a},$$ 
but then 
\begin{align*}
S_n(q^{a-i-1}+1) & \equiv S_n(q^{a-i-1}) + (q^{a-i-1}+1)^n \\
& \equiv 0 + 1^n + nq^{a-i-1} + \sum_{k=2}^n \binom{n}{k}(q^{a-i-1})^k 1^{n-k} \\
& \equiv 1^n + nq^{a-i-1} + 0 \qquad \text{(by Lemma \ref{div-by-qpowers-lemma})} \\ 
& \not\equiv 1^n \qquad \text{(since $q^{i+1}\notdivides n$)} \\
& \equiv S_n(1) \pmod{q^a},
\end{align*}
which is a contradiction. 

Thus we have shown that for $0\le i\le a-2$, if $q-1\notdivides n$, and $\nu_q(n) = i$, then $\ell(q^a,n)$ is a factor of $q^{a-i}$ but not a factor of $q^{a-i-1}$. Therefore, $\ell(q^a,n)=q^{a-i}$. 

In the last case ($q-1\notdivides n$ and $q^{a-1}\divides n$), we have shown above that $\ell(q^a,n)$ is a factor of $q$. However,  
$$S_n(1)\equiv 1 \not\equiv 0 \equiv S_n(q) \pmod{q^a},$$
so $\ell(q^a,n)\neq 1$. Therefore, $\ell(q^a,n)=q$. 
\end{enumerate}
\end{proof}

\begin{section}{Acknowledgment}
The authors would like to thank the College of Science and Mathematics at California State University, Fresno for supporting this work.
\end{section}

\bigskip

\hrule

\bigskip

\noindent 2010 \emph{Mathematics Subject Classification} Primary 11A07; Secondary 11B50, 11A25, 11B83.\\
\emph{Keywords:} number theory, power sum, congruence, periodicity, period, Euler phi-function.

\bigskip 

\hrule 

\bigskip

\noindent (Concerned with sequences \href{http://www.oeis.org/A000010}{A000010}, \href{http://www.oeis.org/A000217}{A000217}, \href{http://www.oeis.org/A026729}{A026729}, \href{http://www.oeis.org/A027641}{A027641}, and \href{http://www.oeis.org/A027642}{A027642}.) 


\begin{thebibliography}{1}

\bibitem{Beery} J. Beery, Sums of powers of positive integers,  \textit{Convergence} {\bf 6} (2009), available at \url{http://www.maa.org/press/periodicals/convergence/sums-of-powers-of-positive-integers-introduction}.

\bibitem{C} L. Carlitz, The Staudt-Clausen theorem, \textit{Math. Mag.} {\bf 34} (1961), 131--146.

\bibitem{Conway} J. H. Conway and R. K. Guy, \textit{The Book of Numbers}, Springer-Verlag, 1996.

\bibitem{HW} G. H. Hardy and E. M. Wright, \textit{An Introduction to the Theory of Numbers,} Oxford University Press, 6th edition, 2008. 

\bibitem{Knuth} D. E. Knuth, Johann Faulhaber and sums of powers, \textit{Math. Comp.} {\bf 61} (1993), 277--294.

\bibitem{MS} K. MacMillan and J. Sondow, Proofs of power sum and binomial coefficient congruences via Pascal's identity, \textit{Amer. Math. Monthly} {\bf 118} (2011), 549--551.

\bibitem{M} P. Moree, A top hat for Moser's four mathemagical rabbits, \textit{Amer. Math. Monthly} {\bf 118} (2011), 364--370. 

\bibitem{M2} P. Moree, Moser's mathemagical work on the equation
$1^k + 2^k + \dots + (m-1)^k = m^k$, \textit{Rocky Mountain J. Math.} {\bf 43} (2013), 1707--1737. 

\bibitem{Mo} L. Moser, On the Diophantine equation $1^n+2^n+\dots +(m-1)^n=m^n$, \textit{Scripta Math.} {\bf 19} (1953), 84--88.

\bibitem{NMA} N. J. Newsome, M. S. Nogin, and A. H. Sabuwala, A proof of symmetry of the power sum polynomials using a novel Bernoulli number identity, \textit{J. Integer Seq.} {\bf 20} (2017), \href{https://cs.uwaterloo.ca/journals/JIS/VOL20/Sabuwala/sabu4.pdf}{Article 17.6.6}.

\bibitem{R} R. Rado, A new proof of a theorem of V. Staudt, 
\textit{J. of London Math. Soc.} {\bf 9} (1934), 85--88. 

\end{thebibliography}
\end{document}